\documentclass{amsart}
\usepackage{amssymb,amsmath}
\headsep=1cm \numberwithin{equation}{section}
\newtheorem{thm}{Theorem}[section]
\newtheorem{cor}[thm]{Corollary}
\newtheorem{lem}[thm]{Lemma}
\newtheorem{prop}[thm]{Proposition}

\newcommand{\Ann}{\mbox{Ann}\,}
\newcommand{\coker}{\mbox{Coker}\,}
\newcommand{\Hom}{\mbox{Hom}\,}
\newcommand{\Ext}{\mbox{Ext}\,}

\newcommand{\Spec}{\mbox{Spec}\,}
\newcommand{\Max}{\mbox{Max}\,}
\newcommand{\Ker}{\mbox{Ker}\,}
\newcommand{\Ass}{\mbox{Ass}\,}
\newcommand{\Assh}{\mbox{Assh}\,}

\newcommand{\Supp}{\mbox{Supp}\,}

\renewcommand{\dim}{\mbox{dim}\,}
\renewcommand{\Im}{\mbox{Im}\,}

\newcommand{\Min}{\mbox{Min}\,}

\newcommand{\h}{\mbox{ht}\,}

\renewcommand{\H}{\mbox{H}}
\newcommand{\V}{\mbox{V}}

\newcommand{\fa}{\mathfrak{a}}
\newcommand{\fb}{\mathfrak{b}}

\newcommand{\fm}{\mathfrak{m}}
\newcommand{\fp}{\mathfrak{p}}
\newcommand{\fq}{\mathfrak{q}}

\begin{document}

\bibliographystyle{amsplain}

\author{Mohammad T. Dibaei}
\address{ Mohammad T. Dibaei\\Faculty of Mathematical Sciences, Tarbiat Moallem
University, Tehran, Iran, and Institute for Theoretical Physics and
Mathematics (IPM), Tehran, Iran.}

\email{dibaeimt@ipm.ir}

\author{Raheleh Jafari}
\address{Raheleh Jafari\\Faculty of Mathematical Sciences, Tarbiat Moallem
University, Tehran, Iran.}

\email{jafarirahele@yahoo.com}

\keywords{Cousin complexes, local cohomology\\
The research of the first author was in part supported from IPM (No.
86130117).}

\subjclass[2000]{13D25; 13D45; 13C14}

\title[Finite Cousin
complexes] {Modules with finite Cousin cohomologies have\\ uniform
local cohomological annihilators}

%%% ----------------------------------------------------------------------

\begin{abstract}
 Let $A$ be a Noetherian ring. It is shown that any finite
 $A$--module $M$ of finite Krull dimension with finite Cousin complex
 cohomologies has a uniform local cohomological
 annihilator. The converse is also true for a finite module
 $M$ satisfying $(S_2)$ which is over a local ring with Cohen--Macaulay formal fibres.

\end{abstract}

%%% ----------------------------------------------------------------------
\maketitle
%%% ----------------------------------------------------------------------

\section{Introduction}
Throughout let $A$ denote a commutative Noetherian ring and $M$ a
finite (i.e. finitely generated) $A$-module. Recall that an
$A$--module $M$ is called {\it equidimensional} (or unmixed) if
$\Min_A(M)= \Assh_A(M)$ (i.e. for each minimal prime $\fp$ of
$\Supp_A(M)$, $\dim_A(M)= \dim(A/\fp))$. For an ideal $\fa$ of $A$,
write $\H_\fa^i(M)$ for the $i$th local cohomology module of $M$
with support in $V(\fa)= \{\fp\in\Spec(A): \fp\supseteq \fa\}$. An
element $x\in A$ is called a {\it uniform local cohomological
annihilator} of $M$ if $x\in A\setminus \cup_{\fp\in\Min_A(M)}\fp$
and for each maximal ideal $\fm$ of $A$, $x\H_\fm^i(M)= 0$ for all
$i< \dim_{A_\fm}(M_\fm)$. The existence of a local cohomological
annihilator is studied by Hochster and Huneke \cite{HH2} and proved
its importance for the existence of big Cohen--Macaulay algebras and
a uniform Artin--Rees theorem \cite{Hu}.

In \cite{Z}, Zhou studied rings with a uniform local cohomological
annihilator. Hochster and Huneke, in \cite{HH1}, proved that if $A$
is locally equidimensional (i.e. $A_\fm$ is equidimensional for
every maximal ideal $\fm$ of $A$) and is a homomorphic image of a
Gorenstein ring of finite dimension, then $A$ has a strong uniform
local cohomological annihilator ( i.e. $A$ has an element which is a
uniform local cohomological annihilator of $A_\fp$ for each
$\fp\in\Spec(A)$). In \cite{Z}, Zhou showed that if a locally
equidimensional ring $A$ of positive dimension is a homomorphic
image of a Cohen--Macaulay ring of finite dimension (or an excellent
local ring), then $A$ has a uniform local
cohomological annihilator.\\

Cousin complexes were introduced by Hartshorne in \cite{H} and have
a commutative algebra analogue given by Sharp in \cite{S1}.
Recently, Cousin complexes have been studied by several authors. In
\cite{D}, \cite{DT}, and \cite{K}, Dibaei, Tousi, and Kawasaki
studied finite Cousin complexes (i.e. the Cousin complexes with
finitely generated cohomologies). In \cite[Proposition 9.3.5]{LNS},
Lipman, Nayak, and Sastry generalized
these results to complexes on formal schemes.\\

In section 2, it is proved that any finite $A$--module of finite
Krull dimension with finite Cousin complex cohomologies has a
uniform local cohomological annihilator (Theorem 2.7). As a result
it follows that if $(A, \fm)$ is local, satisfies Serre's condition
$(S_2)$, and such that all of its fibres of
$A\longrightarrow\widehat{A}$ are Cohen--Macaulay, then $A$ has a
uniform local cohomological annihilator (Corollary 2.10). For a
finite module $M$ over a local ring $(A, \fm)$ satisfying $(S_2)$
and with Cohen Macaulay formal fibres, it is proved that the
following conditions are equivalent: (i) $\widehat{M}$, the
completion of $M$ with respect to $\fm$--adic topology, is
equidimensional; (ii) $\mathcal{C}_A(M)$, the Cousin complex of $M$
is finite; (iii) $M$ has a uniform local cohomological annihilator
(Theorem 2.13).

In section 3, for certain modules $M$, the relationship between the
cohomology modules of the Cousin complex of $M$ and the local
cohomology modules of $M$ with respect to an arbitrary ideal of $A$
is studied. It is shown that the $M$--height of $\fa$ is equal to
the infimum of numbers $r$ for which $0:_A \H_\fa^r(M)$ does not
contain the product of all the annihilators of the Cousin
cohomologies of $M$ (Theorem 3.2).
\section{Cousin complexes}
%2
Let $M$ be an $A$--module and let $\mathcal{H}= \{ H_i: i\geq 0\}$
be the family of subsets of $\Supp_A(M)$ with $H_i=
\{\fp\in\Supp_A(M): \dim_{A_\fp}(M_{\fp})\geq i\}$. The family
$\mathcal{H}$ is called the $M$--height filtration of
$\Supp_A(M)$. Define the Cousin complex of $M$ as the complex
\begin{equation} \mathcal{C}_A(M):
0\overset{d^{-2}}{\longrightarrow}M^{-1}\overset{d^{-1}}{\longrightarrow}
M^0\overset{d^{0}}{\longrightarrow}M^1\overset{d^{1}}{\longrightarrow}
\cdots\overset{d^{i-1}}{\longrightarrow}M^i\overset{d^{i}}{\longrightarrow}
M^{i+1}\longrightarrow\cdots, \tag{*}\end{equation} where $M^{-1}=
M$, $M^i= \underset{\fp\in H_i\setminus H_{i+1}}{\oplus}(\coker
d^{i-2})_\fp$ for $i>-1$. The homomorphism $d^i: M^i\longrightarrow
M^{i+1}$ has the following property: for $m\in M^i$ and $\fp\in
H_i\setminus H_{i+1}$, the component of $d^i(m)$ in $(\coker
d^{i-1})_\fp$ is $\overline{m}/1$, where  $\bar{} :
M^i\longrightarrow \coker d^{i-1}$ is the natural map (see \cite{S1} for details).\\

Throughout, for the Cousin complex (*), we use the following
notations:
$$K^i:= \Ker d^i, D^i:= \Im d^{i-1}, H^i:= K^i/D^i, i=-1, 0,
\cdots.$$ We call the Cousin complex $\mathcal{C}_A(M)$ {\it finite}
if, for each $i$, the cohomology module $H^i$ is finite. Recall that
for an ideal $\fa$ of $A$ and an $A$--module $M$ , the $M$--height
of $\fa$ is defined by $\h_M (\fa):= \inf\{\dim M_\fp :
\fp\in\Supp_A(M)\cap \V(\fa)\}$. Note that $\h_M (\fa)\geq 0$
whenever $M\not= \fa M$. If $M$ is finitely generated then
$\h_M(\fa)= \h(\frac{\fa+I}{I})$, where $I= \Ann_A(M)$.

 We begin by the following lemma which for the first part we adopt
the argument in \cite[Theorem]{S2}.
\begin{lem}
%2.1
Let $M$ be an $A$-module. For any integer $k$ with $0\leq k < \h_M
(\fa)$, the following statements are true.
\begin{itemize}
\item[(a)] $\H_\fa^s(M^k)= 0$ for all integers $s\geq 0$.
\item[(b)] $\Ext_A^s(A/\fa, M^k)= 0$ for all integers $s\geq 0$.
\end{itemize}
\end{lem}
\begin{proof}
(a). Set $C_{k-1}:= \coker d^{k-2}= M^{k-1}/D^{k-1}$ so that $M^k=
\underset{\underset{\h_M (\fp)= k}{\fp\in\Supp_A(M)}}{\oplus}
(C_{k-1})_\fp$. For each $k< \h_M (\fa)$ and each $\fp\in\Supp_A(M)$
with $\h_M (\fp)= k$, there exists an element $x\in
\fa\setminus\fp$. Thus the multiplication map
$(C_{k-1})_\fp\overset{x}{\longrightarrow} (C_{k-1})_\fp$ is an
automorphism and so the multiplication map
$\H_\fa^s((C_{k-1})_\fp)\overset{x}{\longrightarrow}
\H_\fa^s((C_{k-1})_\fp)$ is also an automorphism for all integers
$s$. One may then conclude that $\H_\fa^s((C_{k-1})_\fp)= 0$. Now,
from additivity of local cohomology functors, it follows that
$\H_\fa^s(M^k)= 0$.\\

(b). Assume in general that $N$ is an $A$--module such that
$\H_\fa^s(N)= 0$ for all $s\geq 0$. We show, by induction on $i,
i\geq 0$, that $\Ext_A^i(A/\fa, N)= 0$. For $i= 0$, one has
$\Hom_A(A/\fa, N)= \Hom_A(A/\fa, \H_\fa^0(N))$ which is zero. Assume
that $i> 0$ and the claim is true for any such module $N$ and all
$j\leq i-1$. Choose $E$ to be an injective hull of $N$ and consider
the exact sequence $0\longrightarrow N\longrightarrow
E\longrightarrow N'\longrightarrow 0$, where $N'= E/N$. As
$\H_\fa^0(E)= 0$, it follows that  $\H_\fa^s(N')= 0$ for all $s\geq
0$. Thus $\Ext_A^{i-1}(A/\fa, N')= 0$, by our induction hypothesis.
As, by the above exact sequence $\Ext_A^{i-1}(A/\fa,
N')\cong\Ext_A^{i}(A/\fa, N)$, the result follows.
\end{proof}

 The following technical result is important
for the rest of the paper.
\begin{prop}
%2.2
Let $M$ be an $A$--module and let $\fa$ be an ideal of $A$ such that
$\fa M\not = M$. Then, for each non--negative integer $r$ with $r<
\h_M (\fa)$,
$$\prod_{i= 0}^r(0 :_A \Ext_A^{r-i}(A/\fa, H^{i-1}))\subseteq 0 :_A \Ext_A^r(A/\fa, M).$$
Here $\prod$ is used for product of ideals.
\end{prop}
\begin{proof}
 For each $j\geq -1$, there are the natural exact
sequences
\begin{equation} 0\longrightarrow M^{j-1}/K^{j-1}\longrightarrow
M^j\longrightarrow M^j/D^j\longrightarrow 0, \tag{1}
\end{equation}
\begin{equation} 0\longrightarrow H^{j-1}\longrightarrow
M^{j-1}/D^{j-1}\longrightarrow M^{j-1}/K^{j-1}\longrightarrow 0.
\tag{2}
\end{equation}
Let $0\leq r< \h_M (\fa)$.

We prove by induction on $j$, $0\leq j\leq r$, that
\begin{equation} \prod_{i= 0}^j (0 :_A
\Ext_A^{r-i}(A/\fa, H^{i-1}))\cdot(0 :_A\Ext_A^{r-j}(A/\fa,
M^{j-1}/K^{j-1}))\subseteq 0 :_A \Ext_A^r(A/\fa, M). \tag{3}
\end{equation} In case $j= 0$, the exact sequence (2) implies the
exact sequence
$$\Ext_A^r(A/\fa, H^{-1})\longrightarrow \Ext_A^r(A/\fa, M)\longrightarrow
\Ext_A^r(A/\fa, M^{-1}/K^{-1})$$ so that $$(0 :_A \Ext_A^r(A/\fa,
H^{-1}))\cdot (0 :_A \Ext_A^r(A/\fa, M^{-1}/K^{-1}))\subseteq 0 :_A
\Ext_A^r(A/\fa, M)$$ and thus the case $j= 0$ is justified.

Assume that $0\leq j< r$ and formula (3) is settled for $j$.
Therefore, by Lemma 2.1 (b), formula (1) implies that
\begin{equation}
\Ext_A^{r-j}(A/\fa, M^{j-1}/K^{j-1})\cong \Ext_A^{r-j-1}(A/\fa,
M^j/D^j). \tag{4}
\end{equation}
On the other hand the exact sequence (2) implies the exact sequence
$$
\Ext_A^{r-j-1}(A/\fa, H^j)\longrightarrow \Ext_A^{r-j-1}(A/\fa,
M^j/D^j)\longrightarrow \Ext_A^{r-j-1}(A/\fa, M^j/K^j),
$$
from which it follows that
\begin{equation}
(0:_A \Ext_A^{r-j-1}(A/\fa, H^j))\cdot(0 :_A \Ext_A^{r-j-1}(A/\fa,
\frac{M^j}{K^j}))\subseteq  \hfill 0 :_A \Ext_A^{r-j-1}(A/\fa,
\frac{M^j}{D^j}). \tag{5}
\end{equation}
Now (4) and (5) imply that
\begin{equation}
(0:_A \Ext_A^{r-j-1}(A/\fa, H^j))\cdot(0 :_A \Ext_A^{r-j-1}(A/\fa,
\frac{M^j}{K^j}))\subseteq 0 :_A \Ext_A^{r-j}(A/\fa,
\frac{M^{j-1}}{K^{j-1}}). \tag{6}
\end{equation}
From (6), it follows that
$$ \prod_{i= 0}^{j+1}(0 :_A \Ext_A^{r-i}(\frac{A}{\fa}, H^{i-1}))\cdot
(0:_A\Ext_A^{r-j-1}(\frac{A}{\fa}, \frac{M^j}{K^j})) =$$
$${\prod}_{i= 0}^{j}(0 :_A \Ext_A^{r-i}(\frac{A}{\fa}, H^{i-1}))\cdot (0 :_A
\Ext_A^{r-j-1}(\frac{A}{\fa}, H^{j}))\cdot(0 :_A
\Ext_A^{r-j-1}(\frac{A}{\fa}, \frac{M^j}{K^j})) \subseteq
$$ $$\prod_{i= 0}^{j}(0 :_A \Ext_A^{r-i}(\frac{A}{\fa}, H^{i-1}))\cdot (0 :_A
\Ext_A^{r-j}(\frac{A}{\fa}, M^{j-1}/K^{j-1})),$$

 \noindent and, by the induction hypothesis (3), it follows that  $$\prod_{i= 0}^{j+1}
 (0 :_A \Ext_A^{r-i}(A/\fa, H^{i-1}))\cdot
(0:_A\Ext_A^{r-j-1}(A/\fa, M^j/K^j))\subseteq 0 :_A \Ext_A^r(A/\fa,
M).$$

 This is the end of the induction argument. Putting $j= r$ in
(3) gives the result, because $\Ext_A^0(A/\fa, M^r)= 0$ by Lemma 2.1
(b) and, as by (1) for $j= r$ there is an embedding $\Ext_A^0(A/\fa,
M^{r-1}/K^{r-1})\hookrightarrow \Ext_A^0(A/\fa, M^r)$, it follows
that $\Ext_A^0(A/\fa, M^{r-1}/K^{r-1})= 0.$
\end{proof}
An immediate corollary to the above result is the following.
\begin{cor}
%2.3
 Assume that $M$ is a finite $A$--module and that $\fa$ is an
ideal of $A$ such that $\fa M\not= M$. Then, for each integer $r$
with $0\leq r< \h_M (\fa)$, $$\prod_{i= -1}^{r-1}(0 :_A
H^i)\subseteq\cap_{i= 0}^r(0 :_A \Ext_A^i(A/\fa, M)).$$
\end{cor}
\begin{proof}
It follows by Proposition 2.2 and the fact that the extension
functors are linear.
\end{proof}
\begin{cor}
%2.4
 Let $M$ be a finite $A$--module
of dimension $n$ and let $\fa$ be an ideal of $A$ such that $\fa
M\not= M$. Assume that $x$ is an element of $A$ such that $xH^i= 0$
for all $i$. Then $x^n$ annihilates all the modules $\Ext_A^r(A/\fa,
M), r= 0, 1, \cdots, \h_M(\fa)-1$ for all ideals $\fa$ of $A$.
\end{cor}
\begin{proof}
It follows clearly from Corollary 2.3
\end{proof}
The following lemma states an easy but essential property of
annihilators of Cousin cohomologies.
\begin{lem}
%2.5
Assume that $M$ is a finite $A$--module of finite $\dim_A(M)= n$
and that $\mathcal{C}_A(M)$ is finite, then  $\cap_{i\geq -1}(0
:_A H^i)\not\subseteq \cup_{\fp\in\Min_A(M)}\fp$.
\end{lem}
\begin{proof}
By \cite[(2.7), vii]{S1}, $\V(0:_A H^i)= \Supp_A(H^i)\subseteq
\{\fp\in\Supp_A(M) : \dim_{A_\fp}(M_\fp)\geq i+2\}$ for all $i\geq
-1$. Hence $(0 :_A H^i)\not\subseteq \cup_{\fp\in\Min_A(M)}\fp$. Now
Prime Avoidance Theorem implies that $\cap_{i\geq -1}(0 :_A
H^i)\not\subseteq \cup_{\fp\in\Min_A(M)}\fp$.
\end{proof}
We are now in a position to prove that the modules with finite
Cousin complexes have uniform local cohomological annihilators. But
one can state more.
\begin{prop}
%2.6
Assume that $M$ is a finite $A$--module of finite $\dim_A(M)= n$ and
that $\mathcal{C}_A(M)$ is finite. Then there exists an element
$x\in A\setminus \cup_{\fp\in\Min_A(M)}\fp$ such that
$x\Ext_A^i(A/\fm^j, M)= 0$ for all $i< \h_M(\fm)$, all $j\geq 0$ and
all maximal ideals $\fm$ in $\Supp_A(M)$.
\end{prop}
\begin{proof}
It follows by Lemma 2.5 and Corollary 2.4.
\end{proof}

\begin{thm}
%2.7
Assume that $M$ is a finite $A$--module of finite $\dim_A(M)= n$
and that $\mathcal{C}_A(M)$ is finite, then $M$ has a uniform
local cohomological annihilator.
\end{thm}
\begin{proof}
By Proposition 2.6, there is an element $x\in A\setminus
\cup_{\fp\in\Min_A(M)}\fp$ such that $x\Ext_A^i(A/\fm^j, M)= 0$ for
all $i< \h_M(\fm)$, all $j\geq 0$ and all maximal ideals $\fm$ in
$\Supp_A(M)$. Choose a maximal ideal $\fm$ in $\Supp_A(M)$ and $i<
\h_M(\fm)$. As $x\in\Ann_A(\Ext_A^i(A/\fm^j, M))$ for all $j$, we
have
$x\in\Ann_A(\underset{\underset{j}{\longrightarrow}}{\lim}(\Ext_A^i(A/\fm^j,
M)))$, i.e. $x\H_\fm^i(M)= 0$ for all $i< \h_M(\fm)$.
\end{proof}

\begin{cor}
%2.8
Assume that $A$ has finite dimension and that $ \mathcal{C}_A(A)$ is
finite. Then $A$ has a uniform local cohomological annihilator, and
so $A$ is locally equidimensional and is universally catenary.
\end{cor}
\begin{proof}
It is clear from Theorem 2.6 and \cite[Theorem 2.1]{Z}.
\end{proof}

In \cite[Corollary 3.3]{Z}, Zhou proved that any locally
equidimensional Noetherian ring has a uniform local cohomological
annihilator provided it is a homomorphic image of a Cohen--Macaulay
ring of finite dimension. Here we have the following result:
\begin{cor}
%2.9
Assume that $(A, \fm)$ is local with Cohen--Macaulay formal
fibres. Let $M$ be a finite $A$--module such that it satisfies
($S_2$) and that $\Min_{\widehat{A}}(\widehat{M})=
\Assh_{\widehat{A}}(\widehat{M})$. Then $M$ has a uniform local
cohomological annihilator.
\end{cor}
\begin{proof}
By \cite[Theorem 2.1]{D}, $\mathcal{C}_A(M)$ is finite. Now Theorem
2.7 implies the result.
\end{proof}
\begin{cor}
%2.10
(Compare with \cite[Corollary 3.3 (i)]{Z}). Assume that $(A, \fm)$
is local and that it satisfies $(S_2)$ and all of its formal
fibres are Cohen--Macaulay. Then $A$ has a uniform local
cohomological annihilator.
\end{cor}
\begin{proof}
See \cite[Corollay 2.2]{D}.
\end{proof}

\begin{prop}
%2.11
Let $M$ be a finite $A$--module such that it has a uniform local cohomological annihilator.
Then $M$ is locally equidimensional.
\end{prop}
\begin{proof}
Let $\fm\in\Max\Supp_A(M)$. We will show that
$\dim_{A_\fm}(M_\fm)= \dim A_\fm/\fp A_\fm$ for all $\fp\in\Spec
A$ with $\fp\in\Min_A(M)$ and $\fp\subseteq \fm$. By assumption,
there exists an element $x\in A\setminus\cup_{\fp\in\Min_A(M)}
\fp$ such that $x\H_\fm^i(M)= 0$ for all $i< \dim_{A_\fm}(M_\fm)$.
As $x\in A_\fm\setminus\underset{\fp
A_\fm\in\Min_{A_\fm}(M_\fm)}{\cup} \fp A_\fm$, and
$\H_\fm^i(M)\cong \H_{\fm A_\fm}^i(M_\fm)$ by using the definition
of local cohomology, we may assume that $(A, \fm)$ is local with
the maximal ideal $\fm$ and write $d:= \dim_A(M)$.

Assume, to the contrary, that there exists $\fp\in\Min_A(M)$ with
$c:= \dim A/\fp < d$. Set $S= \{\fq\in\Min_A(M): \dim A/\fq\leq c\}$
and $T= \Ass_A(M)\setminus S$. There exists a submodule $N$ of $M$
such that $\Ass_A(N)= T$ and $\Ass_A(M/N)= S$. Note that
$\dim_A(M/N)= c$ and that $\dim_A(N)= d$. As $\sqrt{0:_A N}=
\cap_{\fq\in T}\fq$, it follows that there exists an element $y\in
0:_A N\setminus\cup_{\fq\in S}\fq$. Thus, trivially, $y\H_\fm^i(N)=
0$ for all $i\geq 0$. The exact sequence $0\longrightarrow
N\longrightarrow M\longrightarrow M/N\longrightarrow 0$ implies the
exact sequence
$\H_\fm^i(M)\longrightarrow\H_\fm^i(M/N)\longrightarrow\H_\fm^{i+1}(N)$.
As $x\H_\fm^i(M)= 0$ for all $i< d$, it follows that
$xy\H_\fm^i(M/N)= 0$ for all $i< d$. In particular,
$xy\H_\fm^c(M/N)= 0$. Thus $xy\in\cap_{\fq\in\Assh_A(M/N)}\fq$ (c.f.
\cite[Proposition 7.2.11 and Theorem 7.3.2]{BS}). Therefore
$xy\in\fp$ by the choice of $\fp$. As $\fp\in S\cap\Min_A(M)$, this
is a contradiction.
\end{proof}

Now we can state the following result which partially extends
Corollary 2.8.
\begin{cor}
%2.12
Let $M$ be a finite $A$--module such that its Cousin complex
$\mathcal{C}_A(M)$ is finite. Then $M$ is locally equidimensional.
\end{cor}
\begin{proof}
The proof is clear from Theorem 2.7 and Proposition 2.11.

\end{proof}
Now it is easy to provide an example of a module whose
Cousin complex has at least one non--finite cohomology.\\

{\bf Example.} Consider a Noetherian local ring $A$ of dimension $d>
2$. Choose any pair of prime ideals $\fp$ and $\fq$ of $A$ with
conditions $\dim A/\fp= 2$, $\dim A/\fq= 1$, and $\fp\not\subseteq
\fq$. Then $\Min_A(A/\fp\fq)= \{\fp, \fq\}$ and so $A/\fp\fq$ is not
an equidimensional $A$--module and thus its
Cousin complex is not finite.\\

We are now ready to present the following result which, for a finite
module $M$, shows connections of finiteness of its Cousin complex,
existence of a uniform local cohomological annihilator for $M$, and
equidimensionality of $\widehat{M}$.
\begin{thm}
%2.13
Let $A$ be a local ring with Cohen--Macaulay formal fibres. Assume
that $M$ is a finite $A$--module which satisfies the condition
$(S_2)$ of Serre. Then the following statements are equivalent.
\begin{itemize}
\item[(i)] $\Min_{\widehat{A}}(\widehat{M})=
\Assh_{\widehat{A}}(\widehat{M})$.
\item[(ii)] The Cousin complex of $M$ is
finite.
\item[(iii)] $M$ has a uniform local cohomological annihilator.
\end{itemize}
\end{thm}
\begin{proof}
(i) $\Rightarrow$ (ii) by \cite[Theorem 2.1]{D}.\\

(ii) $\Rightarrow$ (iii). This is Theorem 2.7.\\

(iii) $\Rightarrow$ (i). There exists an element $x\in A\setminus
\cup_{\fp\in\Min_A(M)}\fp$ such that $x\H_\fm^i(M)= 0$ for all $i<
\dim_A(M)$, and, by artinian--ness of local cohomology modules,
$x\H_{\widehat{\fm}}^i(\widehat{M})= 0$ for all $i<
\dim_{\widehat{A}}(\widehat{M})$. Assume that $\mathcal{Q}$ is an
element of $\Min_{ \widehat{A}}( \widehat{M})$. Note that $0:_A
M\subseteq \mathcal{Q}\cap A$ and, by Going Down Theorem,
$\mathcal{Q}\cap A\in\Min_A(M)$. Hence $x\not\in\mathcal{Q}$.
Therefore $ \widehat{M}$ has a uniform local cohomological
annihilator. Now, Proposition 2.11 implies that
$\Min_{\widehat{A}}(\widehat{M})= \Assh_{\widehat{A}}(\widehat{M})$.
\end{proof}
We end this section by showing that any finite $A$--module $M$ which
has a uniform local cohomological annihilator is universally
catenary, that is the ring $A/(0:_A M)$ is universally catenary.
\begin{thm}
%2.14
Let $M$ be a finite $A$--module that has a uniform local
cohomological annihilator. Then $A/(0:_A M)$ has a uniform local
cohomological annihilator and so $A/(0:_A M)$ is universally
catenary.
\end{thm}
\begin{proof}
By Proposition 2.11, $A/(0:_A M)$ is locally equidimensional. By
\cite[Theorem 3.2]{Z}, it is enough to show that $\frac{A}{0:_A
M}/\frac{\fp}{0:_A M}\cong A/\fp$ has a uniform local cohomological
annihilator for each minimal prime ideal $\fp$ of $M$. We prove it
by using the ideas given in the proof of
\cite[Theorem 3.2]{Z}. \\

Assume that $\fp\in\Min_A(M)$ and that $\fm$ is a maximal ideal
containing $\fp$. As $M_\fp$ is an $A_\fp$--module of finite
length we set $t:= l_{A_\fp}(M_\fp)$. Then there exists a chain of
submodules $0\subset N_1\subset N_2\subset\cdots\subset N_t
\subseteq M$ such that
$$\begin{array}{llllll}
&0\longrightarrow A/\fp\longrightarrow M\longrightarrow
M/N_0\longrightarrow 0,\\
&0\longrightarrow A/\fp\longrightarrow M/N_0\longrightarrow
M/N_1\longrightarrow 0,\\
&\vdots \\
&0\longrightarrow A/\fp\longrightarrow M/N_{t-2}\longrightarrow
M/N_{t-1}\longrightarrow 0,\\
&0\longrightarrow A/\fp\longrightarrow M/N_{t-1}\longrightarrow
M/N_t\longrightarrow 0.\end{array}$$ Since $M_\fm$ is
equidimensional, $\h_M(\fm/\fp)= \h_M(\fm)$. As, by definition of
$t$, $\fp\not\in\Ass_A(M/N_t)$, it follows that
$0:_A(M/N_t)\not\subseteq \fp$. Localizing the above exact
sequences at $\fm$ implies the following exact sequences.
 $$\begin{array}{llllll}
&0\longrightarrow (A/\fp)_\fm\longrightarrow M_\fm\longrightarrow
(M/N_0)_\fm\longrightarrow 0,\\
&0\longrightarrow (A/\fp)_\fm\longrightarrow
(M/N_0)_\fm\longrightarrow
(M/N_1)_\fm\longrightarrow 0,\\
&\vdots \\
&0\longrightarrow (A/\fp)_\fm\longrightarrow
(M/N_{t-2})_\fm\longrightarrow
(M/N_{t-1})_\fm\longrightarrow 0,\\
&0\longrightarrow (A/\fp)_\fm\longrightarrow
(M/N_{t-1})_\fm\longrightarrow 0.\end{array}$$ Choose an element
$y\in0:_A (M/N_t)\setminus \fp$. By assumption, there is an
element $x\in A\setminus\underset{\fq\in\Min_A(M)}{\cup}\fq$ such
that $x\H_{\fm A_\fm}^i(M_\fm)= 0$ for all $i< \h_M(\fm)$. Now,
with a similar technique as in the proof of \cite[Lemma 3.1
(i)]{Z} one can deduce that $(xy)^l\H_\fm^i(A/\fp)_\fm= 0$ for all
$i< \h_M(\fm)$ and for some integer $l> 0$.
\end{proof}
\begin{cor}
%2.15
Let $M$ be a finite $A$--module of finite dimension such its
Cousin complex $\mathcal{C}_A(M)$ is finite. Then the ring $A/0:_A
M$ is universally catenary.
\end{cor}
\begin{proof}
By Theorem 2.7, $M$ has a uniform local cohomological annihilator.
Now, the result follows by Theorem 2.14.
\end{proof}

\section{Height of an ideal}
%3
 As mentioned in Corollary 2.3 and in the proof
of Theorem 2.7, we may write the following corollary.
\begin{cor}
 %3.1
For any finite $A$--module $M$ and any ideal $\fa$ of $A$ with $\fa
M\not=M$,
 $$\underset{-1\leq i}{\prod}(0 :_A
H^i)\subseteq 0 :_A \H_\fa^{\h_M(\fa)- 1}(M).$$
\end{cor}

 We now raise the question that whether it is
possible to improve the upper bound restriction. \\
 \noindent {\bf
Question.} Does the inequality
$$\underset{-1\leq i}{\prod}(0 :_A
H^i)\subseteq 0 :_A \H_\fa^{\h_M(\fa)}(M)$$ hold?

It will be proved that the answer is negative for the class of
finite $A$--modules $M$ with finite Cousin cohomologies. More
precisely,
\begin{thm}
%3.2
Assume that $M$ is a finite $A$--module of finite dimension and that
its Cousin complex $\mathcal{C}_A(M)$ is finite. Then
$$\h_M(\fa)= \inf\{r: \underset{-1\leq i}{\prod}(0 :_A
H^i)\not\subseteq 0 :_A \H_\fa^r(M)\},$$ for all ideals $\fa$ with
$\fa M\not= M$.
\end{thm}
\begin{proof}
By Corollary 2.3, $\underset{i\geq -1}{\prod}(0:_A H^i) \subseteq
0:_A\Ext_A^r(A/\fa^n, M)$ for all $r, 0\leq r< \h_M(\fa)$ and all
$n\geq 0$. Passing to the direct limit, as in the proof of Theorem
2.7, one has  $\underset{i\geq -1}{\prod}(0:_A H^i) \subseteq
0:_A\H_\fa^r(M)$ for all $r< \h_M(\fa)$. Hence we have
$$\h_M(\fa)\leq \inf\{r: \underset{-1\leq i}{\prod}(0 :_A
H^i)\not\subseteq 0 :_A \H_\fa^r(M)\}.$$ Thus it is sufficient to
show that $\underset{-1\leq i}{\prod}(0 :_A H^i)\not\subseteq 0
:_A \H_\fa^{\h_M(\fa)}(M)$.  By Independence Theorem of local
cohomology (c.f. \cite[Theorem 4.2.1]{BS}),
$\H_\fa^{\h_M(\fa)}(M)= \H_\fb^{\h_M(\fb)}(M)$ as
$\overline{A}=A/(0:_A M)$--module, where $\fb= \fa+ 0:_A M/0:_A
M$. Note that $\h_M(\fa)= \h_M(\fb)$ and that
$\mathcal{C}_A(M)\cong\mathcal{C}_{ \overline{A}}(M)$ (see
\cite[Lemma 1.2]{D}).

Hence we may assume that $0:_A M= 0$. Set $h:= \h_M(\fa)$. Let
$x\in 0:_A \H_\fa^h(M)$. As $\fa M\not= M$, there exists a minimal
prime $\fq$ over $\fa$ in $\Supp_A(M)$ such that $\dim (A_\fq)=
\h_M(\fa)$. Hence $x/1\in 0:_{A_\fq} \H_{\fq A_\fq}^h(M_\fq)$.
Thus, by any choice of $\fp A_\fq\in\Assh_{A_\fq}(M_\fq)$ we have
$x/1\in\fp A_\fq$ (see \cite[Proposition 7.2.11(ii) and Theorem
7.3.2]{BS}) and so $x\in\fp$ . Therefore, one has $0:_A
\H_\fa^h(M)\subseteq \underset{\fp\in\Min_A(M)}{\cup} \fp$. On the
other hand, by Lemma 2.5, $\prod_{i\geq -1}(0 :_A
H^i)\not\subseteq \cup_{\fp\in\Min_A(M)}\fp$, from which it
follows that
$$\prod_{i\geq -1}(0 :_A H^i)\not\subseteq 0:_A\H_\fa^h(M).$$
\end{proof}

{\bf Acknowledgment.} The authors are grateful to the referee for
helpful comments.


\begin{thebibliography}{2}

\bibitem{BS}
M.~ P.~ Brodmann, R.~Y.~ Sharp, \emph {Local cohomology: an
algebraic introduction with geometric applications}, Cambridge
Studies in Advanced Mathematics, 60. Cambridge University Press,
Cambridge, 1998.

\bibitem{D}
M.~T.~Dibaei, \emph{A study of Cousin complexes through the
dualizing complexes}, Comm. Alg. \textbf{33} (2005), 119--132.

\bibitem{DT}
M.~T.~Dibaei and M.~Tousi, \emph{The structure of dualizing complex
for a ring which is $(S_2)$}, J. Math. Kyoto University, \textbf{38}
(1998), 503--516.

\bibitem{H}
R.~Hartshorne, \emph{Residues and duality}, Lecture Note in Math.
20, Springer, Berlin, 1966.

\bibitem{HH1}
M.~Hochster and C.~Huneke, \emph{Tight closure, invariant theory,
and the Briancon--Skoda theorem}, J. Amer. Math. Soc. \textbf{3}
(1990), 31--116.

\bibitem{HH2}
M.~Hochster and C.~Huneke, \emph{Infinite integral extensions and
big Cohen--Macaulay algebras}, Ann. of Math. \textbf{3} (1992),
53--89.

\bibitem{Hu} C.~Huneke,
\emph{Uniform bounds in Noetherian rings}, Invent. Math.
\textbf{107} (1992), 203--223.

\bibitem{K}
T.~Kawasaki, \emph{Finiteness of Cousin cohomologies}, Trans. Amer.
Math. Soc., to appear.


\bibitem{LNS}
J.~Lipman, S.~Nayak and P.~Sastry, \emph{Pseudofunctorial behavior
of Cousin complexes on formal schemes}, Contemp. Math.  \textbf{375}
(2005), 3--133.

\bibitem{S1}
R.~Y.~Sharp, \emph{The Cousin complex for a module over a
commutative Noetherian ring}, Math. Z. \textbf{112} (1969),
340--356.

\bibitem{S2}
R.~Y.~Sharp, \emph{Local cohomology and the Cousin complex for a
commutative Noetherian ring}, Math. Z. \textbf{153} (1977), 19--22.

\bibitem{Z}
C.~Zhou, \emph{Uniform annihilators of local cohomology}, Journal
of algebra  \textbf{305} (2006), 585--602.
\end{thebibliography}
\end{document}